\documentclass{article}
\usepackage[left=4cm, right=4cm, top=2.54cm, bottom=2.54cm]{geometry}
\usepackage{amsmath, amssymb, amsthm}
\usepackage[english]{babel} 
\usepackage{mathrsfs} 
\usepackage{cite}
\usepackage{color}
\providecommand{\abs}[1]{\lvert#1\rvert}
\providecommand{\abso}[1]{\bigg{\lvert}#1\bigg{\rvert}}
\providecommand{\absol}[1]{\big{\lvert}#1\big{\rvert}}

\theoremstyle{definition}
\newtheorem{theorem}{Theorem}[section]

\newtheorem{corollary}[theorem]{Corollary}
\newtheorem{lemma}[theorem]{Lemma}

\theoremstyle{definition}  
\newtheorem{defi}[theorem]{Definition}

\theoremstyle{definition}
\newtheorem{remark}[theorem]{Remark}

\numberwithin{equation}{section}

\begin{document} 
\title{Generalized Laplacian decomposition of vector fields on fractal surfaces}

\author{Daniel Gonz\'alez-Campos$^{(1)}$, Marco Antonio P\'erez-de la Rosa$^{(2)}$\\and\\ Juan Bory-Reyes$^{(3)}$}

\date{ \small  $^{(1)}$   Escuela Superior de F\'isica y Matem\'aticas.  Instituto Polit\'ecnico Nacional. CDMX. 07738. M\'exico. \\ E-mail: daniel\_uz13@hotmail.com \\
	$^{(2)}$ Department of Actuarial Sciences, Physics and Mathematics, Universidad de las Am\'ericas Puebla.
	San Andr\'es Cholula, Puebla. 72810. M\'exico. \\ Email: marco.perez@udlap.mx \\
	$^{(3)}$ ESIME-Zacatenco. Instituto Polit\'ecnico Nacional. CDMX. 07738. M\'exico. \\ E-mail: juanboryreyes@yahoo.com }

	\maketitle
	\begin{abstract} 
We consider the behavior of generalized Laplacian vector fields on a Jordan domain of $\mathbb{R}^{3}$ with fractal boundary. Our approach is based on properties of the Teodorescu transform and suitable extension of the vector fields. Specifically, the present article addresses the decomposition problem of a H\"older continuous vector field on the boundary (also called reconstruction problem) into the sum of two generalized Laplacian vector fields in the domain and in the complement of its closure, respectively. In addition, conditions on a H\"older continuous vector field on the boundary to be the trace of a generalized Laplacian vector field in the domain are also established.
\end{abstract}
\small{
} 

\section{Introduction}
Quaternionic analysis is regarded as a broadly accepted branch of classical analysis referring to many different types of extensions of the Cauchy-Riemann equations to the quaternion skew field $\mathbb{H}$, which would somehow resemble the classical complex one-dimensional function theory.

An ordered set of quaternions $\psi:=(\psi_1, \psi_2, \psi_3)\in \mathbb{H}^{3}$, which form an orthonormal (in the usual Euclidean sense) basis in $\mathbb{R}^{3}$ is called a structural $\mathbb{H}$-vector.

The foundation of the so-called $\psi$-hyperholomorphic quaternion valued function theory, see \cite{NM, VSMV, MS} and elsewhere, is that the structural $\mathbb{H}$-vector $\psi$ must be chosen in a way that the factorization of the quaternionic Laplacian holds for $\psi$-Cauchy-Riemann operators. This question goes back at least as far as N\^{o}no's work \cite{Nono1, Nono2}. 

The use of a general orthonormal basis introducing a generalized Moisil-Teodorescu system is the cornerstone of a generalized quaternionic analysis, where the generalized Cauchy-Riemann operator with respect to the standard basis in $\mathbb{R}^3$ are submitted to an orthogonal transformation. Despite the fact that some of the results in the present work can be obtained after the action of an orthogonal transformation on the standard basis; we keep their proofs in the work for the sake of completeness.
 
The $\psi$-hyperholomorphic functions theory by itself is not much of a novelty since it can be reduced by an orthogonal transformation to the standard case. In the face of this, the picture changes entirely by studying some important operators involving a pair of different orthonormal basis. 

Moreover, the possibility to study simultaneously several conventional known theories, which can be embedded into a corresponding version of $\psi$-hyperholomorphic functions theory, again cannot be reduced to the standard context and reveal indeed the relevance of the $\psi$-hyperholomorphic functions theory.
 
The advantageous idea behind the unified study of particular cases of a generalized Moisil-Teodorescu system in $\psi$-hyperholomorphic functions theory simultaneously is considered in the present work.

The special case of structural $\mathbb{H}$-vector $\psi^\theta:=\{\textbf{i},\,  \textbf{i}e^{\textbf{i}\theta}\textbf{j},\, e^{\textbf{i}\theta}\textbf{j}\}$ for $\theta\in[0,2\pi)$ fixed and its associated $\psi^\theta$-Cauchy-Riemann operator 
	\begin{equation*}
	{^{\psi^{\theta}}}D:=\displaystyle\frac{\partial}{\partial x_{1}}\textbf{i}+\frac{\partial}{\partial x_{2}}\textbf{i}e^{\textbf{i}\theta}\textbf{j}+\frac{\partial}{\partial x_{3}} e^{\textbf{i}\theta}\textbf{j},
	\end{equation*} 
are used in \cite{BAPS} to give a quaternionic treatment of inhomogeneous case of the system

	\begin{equation}\label{sedi}
	\left\{
	\begin{array}{rcl}
	-\displaystyle \frac{\partial f_{1}}{\partial x_{1}}+\left(\frac{\partial f_{2}}{\partial x_{2}}-\frac{\partial f_{3}}{\partial x_{3}}\right)\sin\theta-\left(\frac{\partial f_{3}}{\partial x_{2}}+\frac{\partial f_{2}}{\partial x_{3}}\right)\cos\theta & = & 0, 
	\\ {}\\	\displaystyle {\left(\frac{\partial f_{3}}{\partial x_{3}}-\frac{\partial f_{2}}{\partial x_{2}}\right)}\cos\theta-\left(\frac{\partial f_{3}}{\partial x_{2}}+\frac{\partial f_{2}}{\partial x_{3}}\right)\sin\theta & = & 0,
	\\ {}\\ \displaystyle {-\frac{\partial f_{3}}{\partial x_{1}}+\frac{\partial f_{1}}{\partial x_{3}}\sin\theta+\frac{\partial f_{1}}{\partial x_{2}}\cos\theta}  & = & 0, \\ {}\\
	\displaystyle {\frac{\partial f_{2}}{\partial x_{1}}-\frac{\partial f_{1}}{\partial x_{3}}\cos\theta+\frac{\partial f_{1}}{\partial x_{2}}\sin\theta} & = & 0, 
	\end{array}
	\right. 
	\end{equation}
wherein the unknown well-behaved functions $f_m: \Omega \rightarrow \mathbb{C}, m=1,2,3$ are prescribed in an smooth domain $\Omega\subset\mathbb{R}^{3}$. 

From now on, an smooth vector field $\vec{f}=(f_{1}, f_{2}, f_{3})$ that satisfies \eqref{sedi}, will said to be a generalized Laplacian vector field.

We will consider complex quaternionic valued functions (a detailed exposition of notations and definitions will be given in Section 2) to be expressed by
\begin{equation}
	\notag
	f=f_{0}+f_{1}\textbf{i}+f_{2}\textbf{j}+f_{3}\textbf{k},
\end{equation}
where $\textbf{i}$, $\textbf{j}$ and $\textbf{k}$ denote the quaternionic imaginary units.

On the other hand, the one-to-one correspondence 
\begin{equation}\label{corre}
\mathbf{f}=f_1\mathbf{i}+f_2\mathbf{j}+f_3\mathbf{k}\, \longleftrightarrow \vec{f}=(f_{1}, f_{2}, f_{3}) 
\end{equation}
makes it obvious that $\eqref{sedi}$ can be obtained from the classical Moisil-Theodorescu system after the action of some element in $O(3)$ as: 
$${^{\psi^{\theta}}}D[\mathbf{f}]= 0.$$ 

System \eqref{sedi} contains as a particular case the well-known solenoidal and irrotational, or harmonic system of vector fields (see \cite{ABS, ABMP} and the references given there). Indeed, under the correspondence $\mathbf{f}=f_1\mathbf{i}+f_3\mathbf{j}+f_2\mathbf{k}\, \longleftrightarrow \vec{f}=(f_{1}, f_{2}, f_{3})\,$ we have for $\theta=0$:
\begin{equation}\label{equi}
{}{^{\psi^{0}}}D[\mathbf{f}]=0\,\Longleftrightarrow \,
\begin{cases}
\text{div} \vec{f}=0,\cr
\text{rot} \vec{f}=0.
\end{cases}
\end{equation}

Besides, the system \eqref{sedi} includes other partial differential equations systems (see \cite{BAPS} for more details): A particular case of the inhomogeneous Cimmino system (\cite{C}) when one looks for a solution $(f_1,f_2,f_3)$, where each $f_m,\,m=1,2,3$ does not depend on $x_0$. This system is obtained from \eqref{sedi} for $\theta=\frac{\pi}{2}$. Also, an equivalent system to the so-called the Riesz system \cite{Riesz} studied in \cite{Gur, Gur2}, which can be obtained from \eqref{sedi} for $\theta=\pi$ and the convenient embedding in $\mathbb{R}^3$.
	
In order to get more generalized results than those of \cite{ABMP}, it is assumed in this paper that $\Omega\subset \mathbb{R}^{3}$ is a Jordan domain (\cite{HN}) with fractal boundary $\Gamma$ in the Mandelbrot sense, see \cite{FKJ, FJ}.

Let us introduce the temporary notations $\Omega_{+}:=\Omega$ and $\Omega_{-}:=\mathbb{R}^{3}\setminus \{\Omega_{+}\cup\Gamma\}$. We are interested in the following problems: Given a continuous three-dimensional vector field $\vec{f}: \Gamma \rightarrow \mathbb{C}^{3}$:
	\begin{itemize}
		\item [$(I)$] 
		(Problem of reconstruction) Under which conditions can $\vec{f}$ be decomposed on $\Gamma$ into the sum:
		\begin{equation} \label{des}
		\vec{f}(t)=\vec{f}^{+}(t)+\vec{f}^{-}(t),  \quad \forall \, t\in\Gamma,
		\end{equation}
where $\vec{f}^{\pm}$ are extendable to generalized Laplacian vector fields $\vec{F}^{\pm}$ in $\Omega_{\pm}$, with $\vec{F}^{-}(\infty)=0$?
		\item [$(II)$] When $\vec{f}$ is the trace on $\Gamma$ of a generalized Laplacian vector field $\vec{F}^{\pm}$ in $\Omega_{\pm}\cup\Gamma$?
	\end{itemize}
	
In what follows, we deal with problems $(I)$ and $(II)$ using the quaternionic analysis tools and working with $\mathbf{f}$ instead of $\vec{f}$ under the  one-to-one correspondence (\ref{corre}). It will cause no confusion if we call $\mathbf{f}$ also vector field.

In the case of a rectifiable surface $\Gamma$ (the Lipschitz image of some bounded subset of $\mathbb{R}^{2}$) these problems have been investigated in \cite{GPB}. 
\section{Preliminaries.}
\subsection{Basics of $\psi^{\theta}$-hyperholomorphic function theory.}
Let $\mathbb{H}:=\mathbb{H(\mathbb{R})}$ and $\mathbb{H(\mathbb{C})}$ denote the sets of real and complex quaternions respectively. If $a\in\mathbb{H}$ or $a\in\mathbb{H(\mathbb{C})}$, then $a=a_{0}+a_{1}\textbf{i}+a_{2}\textbf{j}+a_{3}\textbf{k}$, where the coefficients $a_{k}\in\mathbb{R}$ if $a\in\mathbb{H}$ and  $a_{k}\in\mathbb{C}$ if $a\in\mathbb{H(\mathbb{C})}$. The symbols
	$\textbf{i}$, $\textbf{j}$ and $\textbf{k}$ denote different imaginary units, i. e. $\textbf{i}^{2}=\textbf{j}^{2}=\textbf{k}^{2}=-1$ and they satisfy the following multiplication rules $\textbf{i}\textbf{j}=-\textbf{j}\textbf{i}=\textbf{k}$; $\textbf{j}\textbf{k}=-\textbf{k}\textbf{j}=\textbf{i}$; $\textbf{k}\textbf{i}=-\textbf{i}\textbf{k}=\textbf{j}$. The unit imaginary $i\in\mathbb{C}$ commutes with every quaternionic unit imaginary. 
	
It is known that $\mathbb{H}$ is a skew-field and $\mathbb{H(\mathbb{C})}$ is an associative, non-commutative complex algebra with zero divisors.
	
If $a\in\mathbb{H}$ or $a\in\mathbb{H(\mathbb{C})}$, $a$ can be represented as $a=a_{0}+\vec{a}$, with $\vec{a}=a_{1}\textbf{i}+a_{2}\textbf{j}+a_{3}\textbf{k}$,
	$\text{Sc}(a):=a_{0}$ is called the scalar part and $\text{Vec}(a):=\vec{a}$ is called the vector part of the quaternion $a$. 
	
Also, if $a\in\mathbb{H(\mathbb{C})}$, $a$ can be represented as $a=\alpha_{1}+i\alpha_{2}$ with $\alpha_{1},\,\alpha_{2}\in\mathbb{H}$.
	
Let $a,\,b\in\mathbb{H(\mathbb{C})}$, the product between these quaternions can be calculated by the formula:
	\begin{equation} \label{pc2}
	ab=a_{0}b_{0}-\langle\vec{a},\vec{b}\rangle+a_{0}\vec{b}+b_{0}\vec{a}+[\vec{a},\vec{b}],
	\end{equation}
where
	\begin{equation} \label{proint}
	\langle\vec{a},\vec{b}\rangle:=\sum_{k=1}^{3} a_{k}b_{k}, \quad
	[\vec{a},\vec{b}]:= \left|\begin{matrix}
	\textbf{i} & \textbf{j} & \textbf{k}\\
	a_{1} & a_{2} & a_{3}\\
	b_{1} & b_{2} & b_{3}
	\end{matrix}\right|.
	\end{equation}
We define the conjugate of $a=a_{0}+\vec{a}\in\mathbb{H(\mathbb{C})}$ by $\overline{a}:=a_{0}-\vec{a}$.
	
The Euclidean norm of a quaternion $a\in\mathbb{H}$ is the number $\abs{a}$ given by: 
	\begin{equation}\label{normar}
	\abs{a}=\sqrt{a\overline{a}}=\sqrt{\overline{a}a}.
	\end{equation}
	
We define the quaternionic norm of $a\in\mathbb{H(\mathbb{C})} $ by:
	\begin{equation}
	\abs{a}_{c}:=\sqrt{{{\abs {a_{0}}}_{\mathbb{C}}}^{2}+{{\abs {a_{1}}}_{\mathbb{C}}}^{2}+{{\abs {a_{2}}}_{\mathbb{C}}}^{2}+{{\abs {a_{3}}}_{\mathbb{C}}}^{2}},
	\end{equation}
where ${\abs {a_{k}}}_{\mathbb{C}}$ denotes the complex norm of each component of the quaternion $a$.The norm of a complex quaternion $a=a_{1}+ia_{2}$ with $a_{1}, a_{2} \in \mathbb{H}$ can be rewritten in the form
	\begin{equation} \label{nc2}
	{\abs{a}_{c}}=\sqrt{\abs{\alpha_{1}}^2+\abs{\alpha_{2}}^2}.
	\end{equation}
	
If $a \in \mathbb{H}$, $b \in \mathbb{H(\mathbb{C})}$, then
	\begin{equation}
	{\abs{ab}}_{c}=\abs{a}{\abs{b}}_{c}.
	\end{equation}
	
If $a\in\mathbb{H(\mathbb{C})}$ is not a zero divisor then $\displaystyle a^{-1}:=\frac{\overline{a}}{a\overline{a}}$ is the inverse of the complex quaternion $a$.	
\begin{subsection}{Notations} 
	\begin{itemize} 
		\item We say that $f:\Omega \rightarrow \mathbb{H(\mathbb{C}})$ has properties in $\Omega$ such as continuity and real differentiability of order $p$ whenever all $f_{j}$  have these properties. These spaces are usually denoted by $C^{p}(\Omega,\, \mathbb{H(\mathbb{C})})$ with $p\in \mathbb{N}\cup\{0\}$.
		\item Throughout this work, $\text{Lip}_{\mu}(\Omega,\, \mathbb{H(\mathbb{C})})$, $0<\mu\leq 1$, denotes the set of H\"older continuous functions $f:\Omega \rightarrow \mathbb{H(\mathbb{C}})$ with H\"older exponent $\mu$. By abuse of notation, when $f_{0}=0$ we write $\mathbf{Lip}_{\mu}(\Omega,\, \mathbb{C}^{3})$ instead of $\text{Lip}_{\mu}(\Omega,\, \mathbb{H(\mathbb{C})})$.
	\end{itemize}
\end{subsection}
	
In this paper, we consider the structural set $\psi^\theta:=\{\textbf{i},\,  \textbf{i}e^{\textbf{i}\theta}\textbf{j},\, e^{\textbf{i}\theta}\textbf{j}\}$ for $\theta\in[0,2\pi)$ fixed, and the associated operators ${^{\psi^\theta}}D$ and $D{^{\psi^\theta}}$ on $C^{1}(\Omega,\, \mathbb{H(\mathbb{C})})$ defined by
\begin{equation}
{^{\psi^{\theta}}}D[f]:=\textbf{i}\frac{\partial f}{\partial x_{1}}+\textbf{i}e^{\textbf{i}\theta}\textbf{j}\frac{\partial f}{\partial x_{2}}+e^{\textbf{i}\theta}\textbf{j}\frac{\partial f}{\partial x_{3}},
\end{equation}
\begin{equation}
D{^{\psi^\theta}}[f]:=\frac{\partial f}{\partial x_{1}}\textbf{i}+\frac{\partial f}{\partial x_{2}}\textbf{i}e^{\textbf{i}\theta}\textbf{j}+\frac{\partial f}{\partial x_{3}} e^{\textbf{i}\theta}\textbf{j},
\end{equation}
which linearize the Laplace operator $\Delta_{\mathbb{R}^{3}}$ in the sense that
\begin{equation}
{^{\psi^{\theta}}}D^{2}= \left[D{^{\psi^\theta}}\right]^{2}=-\Delta_{\mathbb{R}^{3}}.
\end{equation}
All functions belong to $\ker \left({^{\psi^{\theta}}}D\right) := \left\{f : {^{\psi^{\theta}}}D[f]=0\right\}$ are called left-$\psi^{\theta}$-hyperholomorphic in $\Omega$. Similarly, those functions which belong to $\ker \left(D{^{\psi^{\theta}}}\right):= \left\{f : D{^{\psi^{\theta}}}[f]=0\right\}$ will be called right-$\psi^{\theta}$-hyperholomorphic in $\Omega$. For a deeper discussion of the hyperholomorphic function theory we refer the reader to \cite{KVS}.

The function 
\begin{equation} \label{kernel}
\mathscr{K}_{\psi^{\theta}}(x):=-\frac{1}{4\pi}\frac{(x)_{\psi^{\theta}}}{\abs{x}^3}, \quad x\in\mathbb{R}^{3}\setminus\{0\},
\end{equation}
where
\begin{equation}
(x)_{\psi^{\theta}}:=x_{1}\textbf{i}+x_{2} \textbf{i}e^{\textbf{i}\theta}\textbf{j}+x_{3}e^{\textbf{i}\theta}\textbf{j},
\end{equation}
is a both-side-$\psi^{\theta}$-hyperholomorphic fundamental solution of $^{\psi^{\theta}}D$. Observe that $\abs{(x)_{\psi^{\theta}}}=\abs{x}$ for all $ x \in \mathbb{R}^{3}$.		

For $f=f_{0}+\mathbf{f}\in C^1(\Omega,\mathbb{H(\mathbb{C})})$ let us define
\begin{equation}
{^{\psi^{\theta}}}\text{div}[\mathbf{f}]:=\frac{\partial f_{1}}{\partial x_{1}}+\left({\frac{\partial f_{2}}{\partial x_{2}}-\frac{\partial f_{3}}{\partial x_{3}}}\right)\textbf{i}e^{\textbf{i}\theta},
\end{equation}
\begin{equation}
{^{\psi^{\theta}}}\text{grad}[f_{0}]:=\frac{\partial f_{0}}{\partial x_{1}}\textbf{i}+\frac{\partial f_{0}}{\partial x_{2}}\textbf{i}e^{\textbf{i}\theta}\textbf{j}+\frac{\partial f_{0}}{\partial x_{3}}e^{\textbf{i}\theta}\textbf{j},
\end{equation}
\begin{equation}
\begin{split}
{^{\psi^{\theta}}}\text{rot}[\mathbf{f}]:=\left({-\frac{\partial f_{3}}{\partial x_{2}}-\frac{\partial f_{2}}{\partial x_{3}}}\right)e^{\textbf{i}\theta}+\left({-\frac{\partial f_{1}}{\partial x_{3}}\textbf{i}e^{\textbf{i}\theta}-\frac{\partial f_{3}}{\partial x_{1}}}\right)\textbf{j} +\left({\frac{\partial f_{2}}{\partial x_{1}}-\frac{\partial f_{1}}{\partial x_{2}}\textbf{i}e^{\textbf{i}\theta}}\right)\textbf{k}.
\end{split}
\end{equation}
The action of	${^{\psi^{\theta}}}D$ on $f\in C^1(\Omega, \, \mathbb{H(\mathbb{C})})$ yields
\begin{equation}
{^{\psi^{\theta}}}D[f]=-{^{\psi^{\theta}}}\text{div}[\mathbf{f}]+{^{\psi^{\theta}}}\text{grad}[f_{0}]+	{^{\psi^{\theta}}}\text{rot}[\mathbf{f}],
\end{equation}
which implies that $f \in \ker ({^{\psi^{\theta}}}D) $ is equivalent to
\begin{equation} \label{eq1}
-{^{\psi^{\theta}}}\text{div}[\mathbf{f}]+{^{\psi^{\theta}}}\text{grad}[f_{0}]+	{^{\psi^{\theta}}}\text{rot}[\mathbf{f}]=0.
\end{equation}
If $f_{0}=0$, \eqref{eq1} reduces to 
\begin{equation} \label{eq2}
-{^{\psi^{\theta}}}\text{div}[\mathbf{f}]+{^{\psi^{\theta}}}\text{rot}[\mathbf{f}]=0.
\end{equation}
We check at once that \eqref{sedi} is equivalent to \eqref{eq2}.

Similar considerations apply to $D^{\psi^{\theta}}$, for this case one obtains
\begin{equation} \label{eq3}
D^{\psi^{\theta}}[f]=-{^{\overline{\psi^{\theta}}}}\text{div}[\mathbf{f}]+{^{\psi^{\theta}}}\text{grad}[f_{0}]+{^{\overline{\psi^{\theta}}}}\text{rot}[\mathbf{f}],
\end{equation}
where
\begin{equation}
{^{\overline{\psi^{\theta}}}}\text{div}[\mathbf{f}]:=\frac{\partial f_{1}}{\partial x_{1}}+\left({\frac{\partial f_{2}}{\partial x_{2}}-\frac{\partial f_{3}}{\partial x_{3}}}\right)\overline{\textbf{i}e^{\textbf{i}\theta}},
\end{equation}
\begin{equation}
\begin{split}
{^{\overline{\psi^{\theta}}}}\text{rot}[\mathbf{f}]:=\left({-\frac{\partial f_{3} }{\partial x_{2}}-\frac{\partial f_{2}}{\partial x_{3}}}\right) \overline{e^{\textbf{i}\theta}}-{\frac{\partial f_{1}}{\partial x_{3}}\overline{\textbf{i}e^{\textbf{i}\theta}\textbf{j}}+\frac{\partial f_{3}}{\partial x_{1}}}\textbf{j} -\frac{\partial f_{2}}{\partial x_{1}}\textbf{k}-\frac{\partial f_{1}}{\partial x_{2}}\overline{\textbf{i}e^{\textbf{i}\theta}\textbf{k}}.
\end{split}
\end{equation}
If $f_{0}=0$, \eqref{eq3} reduces to
\begin{equation} \label{eq4}
D^{\psi^{\theta}}[f]=-{^{\overline{\psi^{\theta}}}}\text{div}[\mathbf{f}]+{^{\overline{\psi^{\theta}}}}\text{rot}[\mathbf{f}].
\end{equation}	
It follows easily that 
\begin{equation} \label{eq5}
-{^{\overline{\psi^{\theta}}}}\text{div}[\mathbf{f}]+{^{\overline{\psi^{\theta}}}}\text{rot}[\mathbf{f}]=0,
\end{equation}
is also equivalent to \eqref{sedi}.
	\begin{lemma} \label{two-sided} Let $f=f_{0}+\mathbf{f}\in C^{1}(\Omega, \, \mathbb{H(\mathbb{C})})$. Then $f$ is both-side-$\psi^\theta$-hyperholomorphic in $\Omega$ if and only if ${^{\psi^{\theta}}}\text{grad}[f_{0}](x)\equiv 0$ in $\Omega$ and $\mathbf{f}$ is a generalized Laplacian vector field in $\Omega$.
	\begin{proof} 
The proof is based on the fact that \eqref{eq2} and \eqref{eq5} are equivalent to \eqref{sedi}.
	\end{proof}
\end{lemma}

\subsection{Fractal dimension and the Whitney operator}
Let $E$ a subset in $\mathbb{R}^{3}$, we denote by $\mathcal{H}_{\lambda}(E)$ the $\lambda$-Hausdorff measure of $E$ (\cite{GJ}). 

Assume that $E$ is a bounded set, the Hausdorff dimension of $E$ (denoted by $\lambda(E)$) is the infimum $\lambda$ such that $\mathcal{H}_{\lambda}(E)<\infty$.

Frequently, the Minkowski dimension of $E$ (also called box dimension and denoted by $\alpha(E)$) is more appropriate than the Hausdorff dimension to measure the roughness of E (\cite{ABMP,ABS}).

It is known that Minkowski and Hausdorff dimensions can be equal, for example, for rectifiable surfaces (the Lipschitz image of some bounded subset of $\mathbb{R}^{2}$). But in general, if $E$ is a two-dimensional set in $\mathbb{R}^{3}$ 
\begin{equation}
2\leq \lambda(E)\leq \alpha(E)\leq3.
\end{equation}
If $2<\lambda(E)$, $E$ is called a fractal set in the Mandelbrot sense. For more information about the Hausdorff and Minkowski dimension, see \cite{FKJ,FJ}. 

Let $f\in \text{Lip}_{\mu}(\Gamma, \mathbb{H(\mathbb{C})})$, then $f=f_{1}+if_{2}$ with $f_{k}\in \text{Lip}_{\mu}(\Gamma, \mathbb{H(\mathbb{R})})$ and $\mathcal{E}_{0}(f):=\mathcal{E}_{0}(f_{1})+i\mathcal{E}_{0}(f_{2})$. Write
	\begin{equation}
	f^{w}:=\mathcal{X}\mathcal{E}_{0}(f),
	\end{equation}
where $\mathcal{E}_{0}$ is the Whitney operator and $\mathcal{X}$ denotes the characteristic function in $\Omega_{+}\cup\Gamma$.

For completeness, we recall the main lines in the construction of the Whitney decomposition $\mathcal W$ of the Jordan domain $\Omega$ with boundary $\Gamma$ by squares $Q$ of diameter $||Q||_{\mathbb{R}^{3}}$ and the notion of Whitney operator. This can be found in \cite[Ch VI]{SEM}.

Consider the lattice $\mathbb Z^{3}$ in $\mathbb R^{3}$ and the collection of closed unit cubes defined by it; let $\mathcal{M}_1$ be the mesh consisting of those unit cubes having a non-empty intersection with $\Omega$. Then, we recursively define the meshes $\mathcal{M}_k$, $k=2,3,\ldots$, each time bisecting the sides of the cubes of the previous one. The cubes in $\mathcal{M}_k$ thus have side length $2^{-k+1}$ and diameter $||Q||_{\mathbb{R}^{3}} = (\sqrt{3})\, 2^{-k+1}$. Define, for $k=2,3,\ldots$,
\begin{eqnarray*}
	\mathcal{W}^1 & := & \left \{ Q\in \mathcal{M}_1 \, | \, \mbox{$Q$ and every cube of $\mathcal{M}_1$ touching $Q$ are contained in $\Omega$} \right \}, \\
	\mathcal{W}^k & := & \left \{ Q\in \mathcal{M}_k \, | \, \mbox{$Q$ and every cube of $\mathcal{M}_k$ touching $Q$ are contained in $\Omega$} \right .\\
	& & \hspace*{50mm} \left . \mbox{and}\,\not \exists \, Q^\ast \in \mathcal{W}^{k-1}: Q \subset Q^\ast \right \},
\end{eqnarray*}
for which it can be proven that
$$
\Omega = \bigcup_{k=1}^{+\infty} \mathcal{W}^k = \bigcup_{k=1}^{+\infty} \bigcup_{Q \in \mathcal{W}^k} Q \equiv \bigcup_{Q \in \mathcal{W}} Q,
$$
all cubes $Q$ in the Whitney decomposition $\mathcal{W}$ of $\Omega$ having disjoint interiors.

We denote by $Q_{0}$ the unit cube with center at the origin and fix a $C^{\infty}$ function with the properties: $0\leq \varphi \leq 1$; $\varphi(x)=1$ if $x\in Q_{0}$; and $\varphi(x)=0 $ if $x\notin Q^*_{0}$.

Let $\varphi_{k}$ the function $\varphi(x)$ adjusted to the cube $Q_{k}\in\mathcal{W}$, that is
\begin{equation}
\varphi_{k}(x):=\varphi\bigg(\frac{x-x^{k}}{l_{k}}\bigg),
\end{equation}
where $x^{k}$ is the center of $Q_{k}$ and $l_{k}$ the common length of its sides.

Function $\varphi_{k}$ satisfies that $0\leq \varphi_{k} \leq 1$, $\varphi_{k}(x)=1$ if $x\in Q_{k}$ and $\varphi_{k}(x)=0 $ if $x\notin Q^*_{k}$. Let ${\varphi_{k}^*}(x)$ be defined for $x\in \Omega$ by
\begin{equation} \label{pdu}
{\varphi_{k}^*}(x):=\frac{\varphi_{k}(x)}{\Phi (x)},
\end{equation}
with
\begin{equation} 
\Phi(x):=\sum_{k}^{}\varphi_{k}(x)
\end{equation}
and
$\sum_{k}^{}\varphi_{k}^{*}(x)=1$ for $x\in \Omega$.

For each cube $Q_{k}$ let $p_{k}$ be a point fixed in $\Gamma$ such that $dist(Q_{k}, \Gamma)=dist(Q_{k}, p_{k})$. Then the Whitney operator is defined as follows
\begin{equation}
\mathcal{E}_{0}(f)(x):=f(x), \quad \text{if} \quad x\in \Gamma,
\end{equation}
\begin{equation} \label{suma}
\mathcal{E}_{0}(f)(x):=\sum_{k}f(p_{k})\varphi_{k}^{*}(x), \quad \text{if} \quad x\in \Omega. 
\end{equation}
Similar construction may be made for the domain  $\mathbb{R}^{3}\setminus \{\Omega\cup\Gamma\}$.

The operator $\mathcal{E}_{0}$ extends functions $f$ defined in $\Gamma$ to functions defined in $\mathbb{R}^{3}$. Its main properties are given as follows:
\begin{itemize}
\item Assume $f\in\text{Lip}_{\mu}(\Omega\cup\Gamma, \mathbb{H(\mathbb{C})})$. Then $\mathcal{E}_{0}(f) \in \text{Lip}_{\mu}(\mathbb{R}^{3}, \mathbb{H(\mathbb{C})})$ and in fact is $C^{\infty}$ in $\mathbb{R}^{3}\setminus\Gamma$, see \cite[Proposition, pag. 172]{SEM}. 
\item The following quantitative estimate holds (see \cite[(14), pag. 174]{SEM})
\begin{equation}
\absol{\frac{\partial{\mathcal{E}_{0}(f)}}{ { \partial x_{i} } } (x)}\leq c (dist(x, \Gamma))^{\mu-1}, \, \text{for}\, x \in \mathbb{R}^{3}\setminus\Gamma.
\end{equation}
\end{itemize}
It is necessary to go further and to express the essential fact that under some specific relation between $\mu$ and $\alpha(\Gamma)$ we have that
\begin{equation}\label{integrability}
{^{\psi^{\theta}}D}[f^{w}]\in L_{p}(\mathbb{R}^{3}, \mathbb{H(\mathbb{R})})\  \mbox{for}\; \displaystyle p<\frac{3-\alpha(\Gamma)}{1-\mu}.
\end{equation}
This follows in much by the same methods as \cite[Proposition 4.1]{AB}.
\section{Auxiliary results on $\psi^{\theta}$-hyperholomorphic function theory.}
It is a well-known fact that in proving the existence of the boundary value of the Cauchy transform via the Plemelj-Sokhotski formulas, the solvability of the jump problem is an easy task whenever the data is a H\"older continuous function and the boundary of the considered domain is assumed sufficiently smooth. But by far much more subtle is the case where it can be thought of as a fractal surface.  Then the standard method is no longer applicable, and it is necessary to introduce an alternative way of defining Cauchy transform, where a central role is played by the Teodorescu operator involving fractal dimensions. This is the idea behind the proofs of the following auxiliary results.
	
\begin{theorem} \label{thm 6} Let $f\in \text{Lip}_{\mu}(\Gamma,\, \mathbb{H(\mathbb{C})})$, $\displaystyle \frac{\alpha(\Gamma)}{3}<\mu \leq 1$. Then the function $f$ can be represented as $f=\left.F^{+}\right|_{\Gamma}-\left.F^{-}\right|_{\Gamma}$, where $F^{\pm}\in \text{Lip}_{\nu}(\Omega_{\pm}\cup\Gamma)\cap \ker\left( {^{\psi^{\theta}}}D\right)$ for some $\nu<\mu$, $F^{\pm}$ are given by 
    		\begin{equation}\label{tc}
    		F^{\pm}(x):=-{^{\psi^{\theta}}}T\left[{^{\psi^{\theta}}}D[f^{w}]\right](x)+ f^{w}(x), \quad x\in\big({\Omega}_{\pm}\cup\Gamma\big),
    		\end{equation}
    		where
   		\begin{equation}
   		{^{\psi^{\theta}}}T[v](x):=\int_{\Omega_{+}}{\mathscr{K}_{\psi^{\theta}}(x-\xi) \,v(\xi) }\,dm(\xi),  \quad x\in \mathbb{R}^{3}.
   		  \end{equation}
   is the well-defined Teodorescu transform for the $\mathbb{H(\mathbb{C})}$-valued function $v$, see \cite{KVS}. 
	\end{theorem}	
\begin{proof}
	Since ${f}^{w}={f}_{1}^{w}+i {f}_{2}^{w}$ with ${f}_{k}^{w}:\Omega\cup\Gamma\to\mathbb{H}$, $\displaystyle \mu>\frac{\alpha(\Gamma)}{3}$, and by (\ref{integrability})  ${^{\psi^{\theta}}}D[{f}_{k}^{w}]\in L_{p}(\Omega,\, \mathbb{H})$ for some $p\in\left( 3, \,\displaystyle\frac{3-\alpha(\Gamma)}{1-\mu}\right)$. Then the integral on the right side of \eqref{tc} exists and represents a continuous function in the whole $\mathbb{R}^{3}$ (see \cite[Theorem 2.8]{GPB}). Hence, the functions $F^{\pm}$ possess continuous extensions to the closures of the domains $\Omega_{\pm}$ and they satisfy that $\left.F^{+}\right|_{\Gamma}-\left.F^{-}\right|_{\Gamma}=f$. By the property of the Teodorescu operator to still being a right inverse to the Cauchy-Riemann operator (see \cite{KVS}, p. 73), ${^{\psi^{\theta}}}D[F^{+}]=0$ and ${^{\psi^{\theta}}}D[F^{-}]=0$ in the domains $\Omega_{\pm}$, respectively.
\end{proof}  
\begin{remark}
Uniqueness in the statement of Theorem 3.1 could be ensured introducing an additional requirement analogous to that in \cite[Theorem 6.6]{ABJ} 
\end{remark}

In the remainder of this section we assume that $\displaystyle \frac{\alpha(\Gamma)}{3}<\mu \leq 1$.
    The following results are related to the problem of extending $\psi^{\theta}$-hyperholomorphically a $\mathbb{H(\mathbb{C})}$-valued H\"older continuous function.
     \begin{theorem} \label{t1}
     	Let $f\in \text{Lip}_{\mu}(\Gamma,\mathbb{H(\mathbb{C})})$ the trace of $F\in \text{Lip}_{\mu}(\Omega_{+}\cup\Gamma,\mathbb{H(\mathbb{C})})\cap \ker\left(\left.^{\psi^{\theta}}D\right|_{\Omega_{+}}\right).$ Then
     	\begin{equation}\label{c1}
     	{^{\psi^{\theta}}}T\left.\left[{^{\psi^{\theta}}}D[f^{w}]\right]\right|_{\Gamma}=0.
     	\end{equation}
     	Conversely, if \eqref{c1} is satisfied, then $f$ is the trace of  $F\in \text{Lip}_{\nu}(\Omega_{+}\cup\Gamma,\mathbb{H(\mathbb{C})})\cap \ker\left(\left.^{\psi^{\theta}}D\right|_{\Omega_{+}}\right)$ for some $\nu<\mu$.
     	\begin{proof}
     		Sufficiency. As we can write $f=f_{1}+if_{2}$ and $F=F_{1}+iF_{2}$ with $f_{r}\in \text{Lip}_{\mu}(\Gamma,\mathbb{H(\mathbb{R})}), r=1,2$ and $F_{r}\in \text{Lip}_{\nu}(\Omega_{+}\cup\Gamma,\mathbb{H(\mathbb{R})})\cap \ker\left(\left.^{\psi^{\theta}}D\right|_{\Omega_{+}}\right)$. Then $f^{w}=f^{w}_{1}+if^{w}_{2}$ and
     		\begin{equation}
     		{^{\psi^{\theta}}}T\left[{^{\psi^{\theta}}}D[f^{w}]\right]={^{\psi^{\theta}}}T\left[{^{\psi^{\theta}}}D[f_{1}^{w}]\right]+i\;{^{\psi^{\theta}}}T\left[{^{\psi^{\theta}}}D[f_{2}^{w}]\right].
     		\end{equation}
Following \cite[Theorem 3.1]{ABMT}, let $F_{r}^*=f_{r}^{w}-F_{r}$, $\tilde{Q}_{k}$ the union of cubes of the mesh $\mathcal{M}_{k}$ intersecting $\Gamma$,  $\Omega_{k}=\Omega_{+}\setminus \tilde{Q}_{k}$, $\Delta_{k}=\Omega_{+}\setminus\Omega_{k}$ and denote by $\Gamma_{k}$ the boundary of $\Omega_{k}$. Applying the definition of $\alpha(\Gamma)$, given $\varepsilon>0$ there is a constant $C(\varepsilon)$ such that 	$\mathcal{H}^{2}(\Gamma_{k})$ (the Hausdorff measure of $\Gamma_{k}$) is less or equal than $6C(\varepsilon)2^{k(\alpha(\Gamma)-2+\varepsilon)}$.
     		
Since $F_{r}^*\in\text{Lip}_{\mu}(\Gamma,\mathbb{H(\mathbb{C})})$, $F_{r}^*|_{\Gamma}=0$ and any point of $\Gamma_{k}$ is distant by no more than $C_{1}2^{-k}$, then
     		\begin{equation*}
     		\text{max}_{\xi\in\Gamma_{k}}\abs{F_{r}^*(\xi)}\leq C_{2}2^{-\mu k}
     		\end{equation*}
     		where $C_{1}$, $C_{2}$ denoted absolute constants.
     		Therefore, for $x\in\Omega_{-}$, let $s=dist(x,\Gamma)$
     		\begin{equation*}
     		\abso{\int_{\Gamma_{k}}\mathscr{K}_{\psi^{\theta}}(\xi-x){^{\psi^{\theta}}}D[F_{r}^{*}](\xi)dS(\xi)}\leq C_{2}C(\varepsilon)\frac{6}{s^{2}}2^{(\alpha(\Gamma)-2-\mu+\varepsilon)}.
     		\end{equation*}
As $\displaystyle \frac{\alpha(\Gamma)}{3}<\mu \leq 1$ the right-hand side of the previous inequality tends to zero as $k\to \infty$. By the Stokes formula, we have that
\begin{equation*}
     		\begin{split}
     	  &\int_{\Omega_{+}}\mathscr{K}_{\psi^{\theta}}(\xi-x){^{\psi^{\theta}}}D[F_{r}^{*}](\xi)dm(\xi)=\lim_{k\to\infty}\bigg(	     \int_{\Delta_{k}}+\int_{\Omega_{k}}\bigg)\mathscr{K}_{\psi^{\theta}}(\xi-x){^{\psi^{\theta}}}D[F_{r}^{*}](\xi)dm(\xi)\\ &=\lim_{k\to\infty}\bigg(	     \int_{\Delta_{k}}\mathscr{K}_{\psi^{\theta}}(\xi-x){^{\psi^{\theta}}}D[F_{r}^{*}](\xi)dm(\xi)-\int_{\Gamma_{k}}\mathscr{K}_{\psi^{\theta}}(\xi-x){^{\psi^{\theta}}}D[F_{r}^{*}](\xi)dS(\xi)\bigg)=0.
     	  \end{split}
     		\end{equation*}
Then 
\begin{equation}
{^{\psi^{\theta}}}T\left.\left[{^{\psi^{\theta}}}D[f_{r}^{w}]\right]\right|_{\Gamma}={^{\psi^{\theta}}}T\left.\left[{^{\psi^{\theta}}}D[F_{r}]\right]\right|_{\Gamma}=0.
\end{equation}
     		
     		Necessity. If \eqref{c1} is satisfied we have
     		\begin{equation}
     		{^{\psi^{\theta}}}T\left.\left[{^{\psi^{\theta}}}D[f^{w}]\right]\right|_{\Gamma}={^{\psi^{\theta}}}T\left.\left[{^{\psi^{\theta}}}D[f_{1}^{w}]\right]\right|_{\Gamma}+i\;{^{\psi^{\theta}}}T\left.\left[{^{\psi^{\theta}}}D[f_{2}^{w}]\right]\right|_{\Gamma}=0,
     		\end{equation}
     		and we take
     		\begin{equation}
     		\begin{split}
     		F(x):=-{^{\psi^{\theta}}}T\left[{^{\psi^{\theta}}}D[f^{w}]\right](x)+ f^{w}(x), \quad x\in \Omega_{+}\cup\Gamma.
     		\end{split}
     		\end{equation}
     	\end{proof}
     \end{theorem}
In the same manner next theorem can be verified 
    \begin{theorem}
    	Let $f\in \text{Lip}_{\mu}(\Gamma,\mathbb{H(\mathbb{C})})$. If $f$ is the trace of a function $F\in \text{Lip}_{\mu}(\Omega_{-}\cup\Gamma,\mathbb{H(\mathbb{C})})\cap \ker\left(\left.^{\psi^{\theta}}D\right|_{\Omega_{-}}\right)$
    	\begin{equation}\label{c2}
    	{^{\psi^{\theta}}}T\left.\left[{^{\psi^{\theta}}}D[f^{w}]\right]\right|_{\Gamma}=-f.
    	\end{equation}
    	Conversely, if \eqref{c2} is satisfied, then $f$ is the trace of a function $F\in \text{Lip}_{\nu}(\Omega_{-}\cup\Gamma,\mathbb{H(\mathbb{C})})\cap \ker\left(\left.^{\psi^{\theta}}D\right|_{\Omega_{-}}\right)$ for some $\nu<\mu$.
    \end{theorem}
    These two results generalize those of\cite[Theorem 3.1, Theorem 3.2]{ABMT}.   
      \begin{remark}
     	Similar results can be drawn for the case of right $\psi^{\theta}$-hyperholomorphic extensions. The only necessity being to replace in both theorems $\ker\left(\left.^{\psi^{\theta}}D\right|_{\Omega_{\pm}}\right)$ by $\ker\left(\left.D^{\psi^{\theta}}\right|_{\Omega_{\pm}}\right)$ and ${^{\psi^{\theta}}}T\left.\left[{^{\psi^{\theta}}}D[f^{w}]\right]\right|_{\Gamma}$ by $\left[D^{\psi^{\theta}}[f^{w}]\right]\, \left.{{^{\psi^{\theta}}}T}\right|_{\Gamma}$, where for every $\mathbb{H(\mathbb{C})}$-valued function $v$ we have set
     	\begin{equation}
     	[v]\, {{^{\psi^{\theta}}}T}=\int_{\Omega_{+}}{ v(\xi)\, \mathscr{K}_{\psi^{\theta}}(x-\xi) }\,dm(\xi),  \quad x\in \mathbb{R}^{3}.
     	\end{equation}
     	
     	The following theorem presents a result connecting two-sided $\psi^{\theta}$-hyperholomorphicity in the domain $\Omega_{+}$ and it is obtained by application of the previous results
     	
     	\begin{theorem}
     		If $F\in \text{Lip}_{\mu}(\Gamma,\mathbb{H(\mathbb{C})})\cap \ker\left(\left.^{\psi^{\theta}}D\right|_{\Omega_{+}}\right)$ has trace  $\left.F\right|_{\Gamma}=f$, then the following assertions are equivalent:
     		\begin{itemize}
     			\item [1.] F is left and right $\psi^{\theta}$-hyperholomorphic in $\Omega_{+}$,
     			\item [2.] ${^{\psi^{\theta}}}T\left.\left[{^{\psi^{\theta}}}D[f^{w}]\right]\right|_{\Gamma}=	\left[D^{\psi^{\theta}}[f^{w}]\right]\, \left.{{^{\psi^{\theta}}}T}\right|_{\Gamma}$.
     		\end{itemize}
     	\begin{proof}
     		The proof is obtained reasoning as in \cite[Theorem 3.3]{ABMP}.
     	\end{proof}
     	\end{theorem}
     	
     \end{remark}

     \section{Main results}
In this section our main results are stated and proved. They give sufficient conditions for solving the Problems $(I)$ and $(II)$.
      
Let $\mathscr{M}_{\psi^{\theta}}^{*}$ be the subclass of vector fields $\mathbf{f}\in C^{1}(\Omega, \mathbb{C}^{3})\cap\mathbf{Lip}_{\mu}(\Gamma,\, \mathbb{C}^{3})$ defined by
      \begin{equation} \label{set}
      	\mathscr{M}_{\psi^{\theta}}^{*}:=\left\{\mathbf{f}: \int_{\Omega_{+}}{\left\langle\mathscr{K}_{\psi^{\theta}}(x-\xi)\,,\,\mathbf{f}(\xi)\right\rangle}\,dm(\xi)=0, \; x\in\Gamma \right\},
      \end{equation}
      where $m$ denotes the Lebesgue measure in $\mathbb{R}^{3}$. The set $\mathscr{M}_{\psi^{\theta}}^{*}$ can be seen as a fractal version of the corresponding class in \cite{ZMS}, which  can be described in purely physical terms.
     
     \begin{theorem} \label{TH1}
     	Let $\mathbf{f}\in \mathbf{Lip}_{\mu}(\Gamma,\, \mathbb{C}^{3})$ such that $\displaystyle \mu>\frac{\alpha(\Gamma)}{3}$. Then the problem (I) is solvable if 
     	\begin{equation}
     	\begin{split}
     	\text{Vec}\left(-{^{\psi^{\theta}}}\text{div}[\mathbf{f}^{w}]+{^{\psi^{\theta}}}\text{rot}[\mathbf{f}^{w}]\right)&:=\left({\left(\frac{\partial \mathbf{f}^{w}_{3}}{\partial x_{3}}-\frac{\partial \mathbf{f}^{w}_{2}}{\partial x_{2}}\right)}\cos\theta-\left(\frac{\partial \mathbf{f}^{w}_{3}}{\partial x_{2}}+\frac{\partial \mathbf{f}^{w}_{2}}{\partial x_{3}}\right)\sin\theta\right)\textbf{i}\\ & +\left(\displaystyle {-\frac{\partial \mathbf{f}^{w}_{3}}{\partial x_{1}}+\frac{\partial \mathbf{f}^{w}_{1}}{\partial x_{3}}\sin\theta+\frac{\partial \mathbf{f}^{w}_{1}}{\partial x_{2}}\cos\theta}\right)\textbf{j}\\ & +\left(\displaystyle {\frac{\partial \mathbf{f}^{w}_{2}}{\partial x_{1}}-\frac{\partial \mathbf{f}^{w}_{1}}{\partial x_{3}}\cos\theta+\frac{\partial \mathbf{f}^{w}_{1}}{\partial x_{2}}\sin\theta}\right)\textbf{k}\in\mathscr{M}_{\psi^{\theta}}^{*}.
     	\end{split}
     	\end{equation}
     	\begin{proof}
     		It is enough to prove that
     		\begin{equation}
     		\mathbf{F^{\pm}}(x):=-{^{\psi^{\theta}}}T\left[{^{\psi^{\theta}}}D[\mathbf{f}^{w}]\right](x)+ \mathbf{f}^{w}(x), \quad x\in\big({\Omega}_{\pm}\cup\Gamma\big),
     		\end{equation}
     		are vector fields.
     		
     		Observe that
     		\begin{equation}
     		\notag
     		\text{Sc}\left({^{\psi^{\theta}}}T\left[{^{\psi^{\theta}}}D[\mathbf{f}^{w}]\right]\right)(x)=-\int_{\Omega_{+}}{\left\langle \mathscr{K}_{\psi^{\theta}}(x-\xi),\text{Vec}\left(	-{^{\psi^{\theta}}}\text{div}[\mathbf{f}^{w}]+{^{\psi^{\theta}}}\text{rot}[\mathbf{f}^{w}]\right) \right\rangle }\,dm(\xi), \quad x\in \Omega_{\pm},
     		\end{equation}
     		\begin{equation}
     		\notag
     		\Delta\left(\text{Sc}\left({^{\psi^{\theta}}}T\left[{^{\psi^{\theta}}}D[\mathbf{f}^{w}]\right]\right)\right)(x)=0, \quad x\in \Omega_{\pm}
     		\end{equation}
     		and
     		\begin{equation}
     		\notag
     		\text{Sc}\left.\left({^{\psi^{\theta}}}T\left[{^{\psi^{\theta}}}D[\mathbf{f}^{w}]\right]\right)\right|_{\Gamma}=0,
     		\end{equation}
     		because $\text{Vec}\left(	-{^{\psi^{\theta}}}\text{div}[\mathbf{f}^{w}]+{^{\psi^{\theta}}}\text{rot}[\mathbf{f}^{w}]\right)\in\mathscr{M}_{\psi^{\theta}}^{*}$. Therefore $\text{Sc}\left({^{\psi^{\theta}}}T\left[{^{\psi^{\theta}}}D[\mathbf{f}^{w}]\right]\right)\equiv 0$ in $\Omega_{\pm}$. Then $\mathbf{F^{\pm}}(x)$
     are vector fields.
     	\end{proof}
     \end{theorem}
     
     \begin{theorem} \label{TH2}
     	Let $\mathbf{f}$ $\in \mathbf{Lip}_{\mu}(\Gamma,\, \mathbb{C}^{3})$ such that $\displaystyle \mu>\frac{\alpha(\Gamma)}{3}$ and suppose that\\ $\text{Vec}\left(	-{^{\psi^{\theta}}}\text{div}[\mathbf{f}^{w}]+{^{\psi^{\theta}}}\text{rot}[\mathbf{f}^{w}]\right)\in\mathscr{M}_{\psi^{\theta}}^{*}$. If $\bf{f}$ is the trace of a generalized Laplacian vector field in $\mathbf{Lip}_{\mu}(\Omega_{+}\cup\Gamma,\, \mathbb{C}^{3})$, then
     		\begin{equation}\label{c3}
     		\begin{split}
     			&\int_{\Omega_{+}}{\mathscr{K}_{\psi^{\theta}}(t-\xi)\; \text{Sc}\left(-{^{\psi^{\theta}}}\text{div}[\mathbf{f}^{w}]+{^{\psi^{\theta}}}\text{rot}[\mathbf{f}^{w}]\right) }dm(\xi)\\ &=\int_{\Omega_{+}}{\left[\mathscr{K}_{\psi^{\theta}}(t-\xi)\, ,\,\text{Vec}\left(-{^{\psi^{\theta}}}\text{div}[\mathbf{f}^{w}]+{^{\psi^{\theta}}}\text{rot}[\mathbf{f}^{w}]\right)\right] }dm(\xi),  \quad t\in \Gamma,
     		\end{split}
     	\end{equation}
     	where
     	\begin{equation}
     	\begin{split}
     	\text{Sc}\left(-{^{\psi^{\theta}}}\text{div}[\mathbf{f}^{w}]+{^{\psi^{\theta}}}\text{rot}[\mathbf{f}^{w}]\right)&=-\frac{\partial \mathbf{f}^{w}_{1}}{\partial x_{1}}+\left(\frac{\partial \mathbf{f}^{w}_{2}}{\partial x_{2}}-\frac{\partial \mathbf{f}^{w}_{3}}{\partial x_{3}}\right)\sin\theta-\left(\frac{\partial \mathbf{f}^{w}_{3}}{\partial x_{2}}+\frac{\partial \mathbf{f}^{w}_{2}}{\partial x_{3}}\right)\cos\theta.
     	\end{split}
     	\end{equation}
     	Conversely, if \eqref{c3} is satisfied, then $\bf{f}$ is the trace of a generalized Laplacian vector field in $\mathbf{Lip}_{\nu}(\Omega_{+}\cup\Gamma,\, \mathbb{C}^{3})$ for some $\nu<\mu$.
     		\begin{proof}
     		Suppose that  $\mathbf{f}$ $\in \mathbf{Lip}_{\mu}(\Gamma,\, \mathbb{C}^{3})$ is the trace of a generalized Laplacian vector field in $\mathbf{Lip}_{\mu}(\Omega_{+}\cup\Gamma,\, \mathbb{C}^{3})$. Therefore  
     		\begin{equation*}
     		{^{\psi^{\theta}}}T\left.\left[{^{\psi^{\theta}}}D[\mathbf{f}^{w}]\right]\right|_{\Gamma}=0,
     		\end{equation*}
			by Theorem \ref{t1}.
     		Of course
     		\begin{equation*}
     		\begin{split}
     		&\int_{\Omega_{+}}{\mathscr{K}_{\psi^{\theta}}(t-\xi)\; \text{Sc}\left(-{^{\psi^{\theta}}}\text{div}[\mathbf{f}^{w}]+{^{\psi^{\theta}}}\text{rot}[\mathbf{f}^{w}]\right) }\,dm(\xi)\\ &=\int_{\Omega_{+}}{\left[\mathscr{K}_{\psi^{\theta}}(t-\xi)\, ,\,\text{Vec}\left(-{^{\psi^{\theta}}}\text{div}[\mathbf{f}^{w}]+{^{\psi^{\theta}}}\text{rot}[\mathbf{f}^{w}]\right)\right] }\,dm(\xi),  \quad t\in \Gamma,
     		\end{split}
     		\end{equation*} 
				as is easy to check.
				
     		Now, if \eqref{c3} is satisfied. Set
     		\begin{equation}
     		\mathbf{F^{+}}(x):=-{^{\psi^{\theta}}}T\left[{^{\psi^{\theta}}}D[\mathbf{f}^{w}]\right](x)+ \mathbf{f}^{w}(x), \quad x\in\big(\Omega_{+}\cup\Gamma\big).
     		\end{equation}
     		As $\text{Vec}\left(	-{^{\psi^{\theta}}}\text{div}[\mathbf{f}^{w}]+{^{\psi^{\theta}}}\text{rot}[\mathbf{f}^{w}]\right)\in\mathscr{M}_{\psi^{\theta}}^{*}$, $\mathbf{F^{+}}$ is a generalized Laplacian vector field in $\Omega_{+}$. By Theorem \ref{thm 6}, $\left.\mathbf{F^{+}}\right|_{\Gamma}=\mathbf{f}$, which completes the proof.
     	\end{proof}
     \end{theorem}
 
     The method of proof carries to domain $\Omega_{-}$. Indeed, we have
     \begin{theorem} \label{TH3}
     	Let $\mathbf{f}\in \mathbf{Lip}_{\mu}(\Gamma,\, \mathbb{C}^{3})$ such that $\displaystyle \mu>\frac{\alpha(\Gamma)}{3}$ and suppose that\\ $\text{Vec}\left(	-{^{\psi^{\theta}}}\text{div}[\mathbf{f}^{w}]+{^{\psi^{\theta}}}\text{rot}[\mathbf{f}^{w}]\right)\in\mathscr{M}_{\psi^{\theta}}^{*}$. If $\bf{f}$ is the trace of a generalized Laplacian vector field in $\mathbf{Lip}_{\mu}(\Omega_{-}\cup\Gamma,\, \mathbb{C}^{3})$ which vanishes at infinity, then
     \begin{equation}\label{c4}
     \begin{split}
     &\int_{\Omega_{+}}{\mathscr{K}_{\psi^{\theta}}(t-\xi)\; \text{Sc}\left(-{^{\psi^{\theta}}}\text{div}[\mathbf{f}^{w}]+{^{\psi^{\theta}}}\text{rot}[\mathbf{f}^{w}]\right) }\,dm(\xi)\\ &-\int_{\Omega_{+}}{\left[\mathscr{K}_{\psi^{\theta}}(t-\xi)\, ,\,\text{Vec}\left(-{^{\psi^{\theta}}}\text{div}[\mathbf{f}^{w}]+{^{\psi^{\theta}}}\text{rot}[\mathbf{f}^{w}]\right)\right] }\,dm(\xi)=-\mathbf{f}(t),  \quad t\in \Gamma.
     \end{split}
     \end{equation}
     	Conversely, if \eqref{c4} is satisfied, then $\bf{f}$ is the trace of a generalized Laplacian vector field in $\mathbf{Lip}_{\nu}(\Omega_{-}\cup\Gamma,\, \mathbb{C}^{3})$ for some $\nu<\mu$, which vanishes at infinity.
     \end{theorem}
\begin{remark} The mains results of this paper are generalizations of those in \cite{ABMP}, where is considered the operator Moisil-Teodorescu
	\begin{equation}
	D_{MT}:=\textbf{i}\frac{\partial }{\partial x_{1}}+\textbf{j}\frac{\partial }{\partial x_{2}}+\textbf{k}\frac{\partial }{\partial x_{3}}.
	\end{equation}
	Applying the operator $D_{MT}$ to $\mathbf{h}^{w}:=\mathbf{f}^{w}_{1}\textbf{i}+\mathbf{f}^{w}_{2}\textbf{j}+\mathbf{f}^{w}_{3}\textbf{k}\in C^{1}(\Omega, \mathbb{C}^{3})\cap\mathbf{Lip}_{\mu}(\Gamma,\, \mathbb{C}^{3})$ we get
	\begin{equation}
	\begin{split}
	D_{MT}[\mathbf{h}^{w}]&=-\text{div}[\mathbf{h}^{w}]+\text{rot}[\mathbf{h}^{w}]\\ &=-\frac{\partial \mathbf{f}^{w}_{1}}{\partial x_{1}}-\frac{\partial \mathbf{f}^{w}_{2}}{\partial x_{2}}-\frac{\partial \mathbf{f}^{w}_{3}}{\partial x_{3}}+\left(\frac{\partial \mathbf{f}^{w}_{3}}{\partial x_{2}}-\frac{\partial \mathbf{f}^{w}_{2}}{\partial x_{3}}\right)\textbf{i}\\ & +\left(\displaystyle {\frac{\partial \mathbf{f}^{w}_{1}}{\partial x_{3}}-\frac{\partial \mathbf{f}^{w}_{3}}{\partial x_{1}}}\right)\textbf{j}+\left(\displaystyle {\frac{\partial \mathbf{f}^{w}_{2}}{\partial x_{1}}-\frac{\partial \mathbf{f}^{w}_{1}}{\partial x_{2}}}\right)\textbf{k}.
	\end{split}
	\end{equation}
	For abbreviation, we let $D_{MT}[\mathbf{h}^{w}]$ stand for
	\begin{equation}  \label{2}
	\begin{split}
	D_{MT}[\mathbf{h}^{w}]=\left[D_{MT}[\mathbf{h}^{w}]\right]_{0}+\left[D_{MT}[\mathbf{h}^{w}]\right]_{1}\textbf{i} +\left[D_{MT}[\mathbf{h}^{w}]\right]_{2}\textbf{j}+\left[D_{MT}[\mathbf{h}^{w}]\right]_{3}\textbf{k}.
	\end{split}
	\end{equation}
On the other hand, setting $\mathbf{f}^{w}:=\mathbf{f}^{w}_{1}\textbf{i}+\mathbf{f}^{w}_{3}\textbf{j}+\mathbf{f}^{w}_{2}\textbf{k}\in C^{1}(\Omega, \mathbb{C}^{3})\cap\mathbf{Lip}_{\mu}(\Gamma,\, \mathbb{C}^{3})$ it follows that
	\begin{equation}
	\begin{split}
	 {^{\psi^{0}}D}[\mathbf{f}^{w}]&=-\frac{\partial \mathbf{f}^{w}_{1}}{\partial x_{1}}-\frac{\partial \mathbf{f}^{w}_{2}}{\partial x_{2}}-\frac{\partial \mathbf{f}^{w}_{3}}{\partial x_{3}}+\left({\frac{\partial \mathbf{f}^{w}_{2}}{\partial x_{3}}-\frac{\partial \mathbf{f}^{w}_{3}}{\partial x_{2}}}\right)\textbf{i}\\ & +\left(\displaystyle {\frac{\partial \mathbf{f}^{w}_{1}}{\partial x_{2}}-\frac{\partial \mathbf{f}^{w}_{2}}{\partial x_{1}}}\right)\textbf{j}+\left(\displaystyle {\frac{\partial \mathbf{f}^{w}_{3}}{\partial x_{1}}-\frac{\partial \mathbf{f}^{w}_{1}}{\partial x_{3}}}\right)\textbf{k}.
	 \end{split}
	 \end{equation}
 The above expression may be written as
 \begin{equation} \label{1}
 \begin{split}
 {^{\psi^{0}}D}[\mathbf{f}^{w}]=\left[{^{\psi^{0}}D}[\mathbf{f}^{w}]\right]_{0}+\left[{^{\psi^{0}}D}[\mathbf{f}^{w}]\right]_{1}\textbf{i}+\left[{^{\psi^{0}}D}[\mathbf{f}^{w}]\right]_{2}\textbf{j}+\left[{^{\psi^{0}}D}[\mathbf{f}^{w}]\right]_{3}\textbf{k}.
 \end{split}
 \end{equation}
It is worth noting that under the correspondence $\left(\mathbf{f}^{w}_{1},\,\mathbf{f}^{w}_{2},\,\mathbf{f}^{w}_{3}\right)\, \leftrightarrow \, \left(\mathbf{f}^{w}_{1},\,\mathbf{f}^{w}_{3},\,\mathbf{f}^{w}_{2}\right)$ we can assert that
\begin{equation}\label{equiv}
D_{MT}[\mathbf{h}^{w}]=0\,\Longleftrightarrow \, {}{^{\psi^{0}}D}[\mathbf{f}^{w}]=0,
\end{equation}
which follow from 
 \begin{align*}
 \left[D_{MT}[\mathbf{h}^{w}]\right]_{0} &=\left[{^{\psi^{0}}D}[\mathbf{f}^{w}]\right]_{0} ,\\
\left[D_{MT}[\mathbf{h}^{w}]\right]_{1} & =- \left[{^{\psi^{0}}D}[\mathbf{f}^{w}]\right]_{1},\\
\left[D_{MT}[\mathbf{h}^{w}]\right]_{2} & =-\left[{^{\psi^{0}}D}[\mathbf{f}^{w}\right]_{3}, \\
 \left[D_{MT}[\mathbf{h}^{w}]\right]_{3} & =-\left[{^{\psi^{0}}D}[\mathbf{f}^{w}]\right]_{2}.
 \end{align*}
\end{remark}
\begin{remark}
	In \cite{ABMP} is defined 
		\begin{equation}
	\mathscr{M}^{*}:=\left\{\mathbf{f}: \frac{1}{4\pi}\int_{\Omega_{+}}{\left\langle  \text{grad}\;\frac{1}{\abs{t-\xi}}\, ,\,\mathbf{f}(\xi)\right\rangle}\,dm(\xi)=0, \, t\in\Gamma \right\}.
	\end{equation}
	For $\mathbf{h}:=\mathbf{f_{1}}\textbf{i}+\mathbf{f_{2}}\textbf{j}+\mathbf{f_{3}}\textbf{k} \in \mathscr{M}^{*}$ it is clear that
	\begin{equation}
	\begin{split}
	 \frac{1}{4\pi}\int_{\Omega_{+}}{\left\langle  \text{grad}\;\frac{1}{\abs{t-\xi}}\, ,\,\mathbf{h}(\xi)\right\rangle}\,dm(\xi)=\int_{\Omega_{+}}{\left\langle\mathscr{K}_{\psi^{0}}(t-\xi)\, ,\,\mathbf{f}(\xi)\right\rangle}\,dm(\xi)=0,
	 \end{split}
	\end{equation}
where $\mathbf{f}:=\mathbf{f}_{1}\textbf{i}+\mathbf{f}_{3}\textbf{j}+\mathbf{f}_{2}\textbf{k} \in \mathscr{M}^{*}_{\psi^{0}}$. Hence 
$$\mathbf{h}:=\mathbf{f}_{1}\textbf{i}+\mathbf{f}_{2}\textbf{j}+\mathbf{f}_{3}\textbf{k}\in \mathscr{M}^{*} \iff \mathbf{f}:=\mathbf{f}_{1}\textbf{i}+\mathbf{f}_{3}\textbf{j}+\mathbf{f}_{2}\textbf{k} \in \mathscr{M}^{*}_{\psi^{0}}.$$
	\end{remark}
From Theorems \ref{TH1}, \ref{TH2}, \ref{TH3} and the previous remarks the followings corollaries are obtained.
\begin{corollary} \cite[Theorem 2.2]{ABMP}.
	Let $\mathbf{f}\in \mathbf{Lip}_{\mu}(\Gamma,\, \mathbb{C}^{3})$ such that $\displaystyle \mu>\frac{\alpha(\Gamma)}{3}$. Then the reconstruction problem for the div-rot system is solvable if $\text{rot}[\mathbf{f}^{w}]\in\mathscr{M}^{*}$.
\end{corollary}

\begin{corollary} \cite[Theorem 2.3]{ABMP}.
	Let $\mathbf{f}\in \mathbf{Lip}_{\mu}(\Gamma,\, \mathbb{C}^{3})$ such that $\displaystyle \mu>\frac{\alpha(\Gamma)}{3}$ and suppose that $\text{rot}[\mathbf{f}^{w}]\in\mathscr{M}^{*}$. If $\bf{f}$ is the trace of a Laplacian vector field in $\mathbf{Lip}_{\mu}(\Omega_{+}\cup\Gamma,\, \mathbb{C}^{3})$, then
	\begin{equation}\label{c31}
		\begin{split}
			&\frac{1}{4\pi}\int_{\Omega_{+}}{ \text{grad}\;\frac{1}{\abs{t-\xi}}\; \text{div}[\mathbf{f}^{w}]}\,dm(\xi)\\ &=\frac{1}{4\pi}\int_{\Omega_{+}}{\left[ \text{grad}\;\frac{1}{\abs{t-\xi}}\,,\, \text{rot}[\mathbf{f}^{w}]\right] }\,dm(\xi),  \quad t\in \Gamma.
		\end{split}
	\end{equation}
	Conversely, if \eqref{c31} is satisfied, then $\bf{f}$ is the trace of a Laplacian vector field in $\mathbf{Lip}_{\nu}(\Omega_{+}\cup\Gamma,\, \mathbb{C}^{3})$ for some $\nu<\mu$.
\end{corollary}

\begin{corollary} \cite[Theorem 2.4]{ABMP}.
	Let $\mathbf{f}\in\mathbf{Lip}_{\mu}(\Gamma,\, \mathbb{C}^{3})$ such that $\displaystyle \mu>\frac{\alpha(\Gamma)}{3}$ and suppose that $\text{rot}[\mathbf{f}^{w}]\in\mathscr{M}^{*}$. If $\bf{f}$ is the trace of a Laplacian vector field in $\mathbf{Lip}_{\mu}(\Omega_{-}\cup\Gamma, \, \mathbb{C}^{3})$ which vanishes at infinity, then
	\begin{equation}\label{c42}
		\begin{split}
			&\frac{1}{4\pi}\int_{\Omega_{+}}{ \text{grad}\;\frac{1}{\abs{t-\xi}}\;\text{div}[\mathbf{f}^{w}]}\,dm(\xi)\\ &-\frac{1}{4\pi}\int_{\Omega_{+}}{\left[ \text{grad}\;\frac{1}{\abs{t-\xi}}\,,\,\text{rot}[\mathbf{f}^{w}]\right] }\,dm(\xi)=-\mathbf{f}(t),  \quad t\in \Gamma.
		\end{split}
	\end{equation}
	Conversely, if \eqref{c42} is satisfied, then $\bf{f}$ is the trace of a Laplacian vector field in $\mathbf{Lip}_{\nu}(\Omega_{-}\cup\Gamma,\, \mathbb{C}^{3})$ for some $\nu<\mu$, which vanishes at infinity.
\end{corollary}

\section*{Acknowledgements}
D. Gonz\'alez-Campos gratefully acknowledges the financial support of the Postgraduate Study Fellowship of the Consejo Nacional de Ciencia y Tecnolog\'ia (CONACYT) (grant number 818693). J. Bory-Reyes and M. A. P\'erez-de la Rosa were partially supported by Instituto Polit\'ecnico Nacional in the framework of SIP programs (SIP20211188) and by Fundaci\'on Universidad de las Am\'ericas Puebla, respectively.

\section*{Appendix. Criteria for the generalized Laplacianness of a vector field}
We continue to assume that $\Omega\subset \mathbb{R}^{3}$ is a Jordan domain with a fractal boundary $\Gamma$. Our interest here is to find necessary and sufficient conditions for the generalized Laplacianness of an vector field $\mathbf{F}\in\mathbf{Lip}_{\nu}(\Omega\cup\Gamma,\, \mathbb{C}^{3})$ in terms of its boundary value $\mathbf {f}:=\left.\mathbf {F}\right|_\Gamma$.
	
The inspiration for the following definition is that in \cite[Definition 2.1]{ARBR}.
	\begin{defi} \label{dtc1}
		Let $\Omega$ a Jordan domain with fractal boundary $\Gamma$. Then we define the Cauchy transform of $\mathbf{f}\in \mathbf{Lip}_{\mu}(\Gamma,\,\mathbb{C}^{3})$ by	
		\begin{equation}\label{tc2}
			K_{\Gamma}^{*}[\mathbf{f}](x):=-{^{\psi^{\theta}}}T\left[{^{\psi^{\theta}}}D[\mathbf{f}^{w}]\right](x)+ \mathbf{f}^{w}(x), \quad x\in \mathbb{R}^{3}\setminus\Gamma.
		\end{equation}
	\end{defi}
Under condition $\displaystyle \frac{\alpha(\Gamma)}{3}<\mu \leq 1$ the Cauchy transform $K_{\Gamma}^{*}[\mathbf{f}]$ has continuous extension to $\Omega\cup\Gamma$ for every vector field $\mathbf {f}\in \mathbf{Lip}_{\mu}(\Gamma,\,\mathbb{C}^{3})$ (take a fresh look at Theorem \ref{thm 6}). On the other hand, using the properties of the Theodorescu operator (see \cite{KVS}, p. 73) we obtain that $K_{\Gamma}^{*}[\bf{f}]$ is left-$\psi^{\theta}$-hyperholomorphic in $\mathbb{R}^{3}\setminus\Gamma$. Note that $K_{\Gamma}^{*}[\mathbf{f}](x)$ vanishes at infinity. 
	
Let us introduce the following fractal version of the Cauchy singular integral operator 
	\begin{equation*}
		\mathcal{S}_{\Gamma}^{*}[\mathbf{f}](x):=2K^{*}_{\Gamma}[\mathbf{f}]^{+}(x)-f(x), \quad x\in\Gamma.
	\end{equation*}
	Here and subsequently, $K^{*}_{\Gamma}[\mathbf{f}]^{+}$ denotes the trace on $\Gamma$ of the continuous extension of $K^{*}_{\Gamma}[\mathbf{f}]$ to $\Omega\cup\Gamma$.
	
Let us now establish and prove the main result of this appendix, which gives necessary and sufficient conditions for the generalized Laplacianness of a vector field in terms of its boundary value.
	\begin{theorem}
		Let $\mathbf{F}\in\mathbf{Lip}_{\mu}(\Omega\cup\Gamma,\mathbb{C}^{3})$ with trace $\mathbf{f}=\left.\mathbf{F}\right|_{\Gamma}$. Then the following sentences are equivalent:
		\begin{itemize}
			\item [(i)] $\mathbf{F}$ is a generalized Laplacian vector field.
			\item [(ii)] $\mathbf{F}$ is harmonic in $\Omega$ and $\mathcal{S}_{\Gamma}^{*}[\mathbf{f}]=\mathbf{f}$.
		\end{itemize}
		\begin{proof} Let $\mathbf{F}^{w}$ be the Whitney extension of $\mathbf{F}$ in $\mathbf{Lip}_{\mu}(\Omega\cup\Gamma,\mathbb{C}^{3})$. Suppose that $\mathbf{F}$ is a generalized Laplacian vector field in $\Omega$. Since ${^{\psi^{\theta}}}D[\mathbf{F}]=0$ in $\Omega$, it follows that $\mathbf{F}$ is harmonic. Also $\mathbf{F}^{w}$ is a Whitney extension of $\mathbf{f},$ i.e. $\mathbf{f}=\left.\mathbf{F}^{w}\right|_{\Gamma}$.
			According to Definition \ref{dtc1}, with $\mathbf{f}^{w}$ replaced by $\mathbf{F}^{w}$, we get
			\begin{equation*}
				K_{\Gamma}^{*}[\mathbf{f}](x)=-\int_{\Omega}{\mathscr{K}_{\psi^{\theta}}(x-\xi)\, {{^{\psi^{\theta}}}D}[\mathbf{F}^{w}](\xi) }\,dm(\xi)+ \mathbf{F}^{w}(x)=\mathbf{F}(x), \quad x\in\Omega,
			\end{equation*}
which imply that $K_{\Gamma}^{*}[\mathbf{f}]^{+}=\mathbf{f}$ and $\mathcal{S}_{\Gamma}^{*}[\mathbf{f}]=\mathbf{f}$.
			
Conversely, assume that $(ii)$ holds and define
			\begin{equation}
				\Psi(x):= \left\{
				\begin{array}{ll}
					K_{\Gamma}^{*}[\mathbf{f}](x),  &  x \in\Omega, \\
					\mathbf{f}(x), &  x \in \Gamma.
				\end{array}
				\right.
			\end{equation}
			Note that $\Psi(x)$ is left-$\psi^{\theta}$-hyperholomorphic function, hence harmonic in $\Omega$. Since $\mathcal{S}_{\Gamma}^{*}[\mathbf{f}]=\mathbf{f}$ in $\Gamma$, it follows that $K_{\Gamma}^{*}[\mathbf{f}]^{+}=\mathbf{f}$. Therefore $K_{\Gamma}^{*}[\mathbf{f}]$ is also continuous on $\Omega\cup\Gamma$.
			
			As $\mathbf{F}-\Psi$ is harmonic in $\Omega$ and $\left.(\mathbf{F}-\Psi)\right|_{\Gamma}=0$ we have that $\mathbf{F}(x)=K_{\Gamma}^{*}[\mathbf{f}](x)$ for all $x \in\Omega,$ which follows from the harmonic maximum principle. Lemma \ref{two-sided} now forces $\mathbf{F}$ to be a generalized Laplacian vector field in $\Omega,$ and the proof is complete.
		\end{proof}
	\end{theorem}

\end{document}